\documentclass[11pt,twoside,reqno]{amsart}
\usepackage{amsmath,amsfonts,amsthm,amssymb}
\usepackage{graphicx,psfrag,overpic}
\usepackage{bm}
\usepackage[all,cmtip,line]{xy}
\usepackage{array,setspace}
\usepackage{color}
\numberwithin{equation}{section}

\usepackage{graphics}
\usepackage{epsfig}
\newtheorem{claim}{\bf \t}[part]

\newtheorem{theorem}{Theorem}[section]

\newtheorem{lemma}{Lemma}[section]
\newtheorem{proposition}[theorem]{Proposition}
\newtheorem{remark}{Remark}[section]

\def\t{\theta}

\newcommand \R{\mathbb{R}}

\def\bt{\begin{theorem}}
	\def\et{\end{theorem}}
\def\ba{\begin{array}}
	\def\ea{\end{array}}
\def\bl{\begin{lemma}}
	\def\el{\end{lemma}}

\def\bes{\begin{eqnarray}}
\def\ees{\end{eqnarray}}

\newcommand{\Curl}{\mathrm{curl}}

\begin{document}
	
	\title[Subsonic and subsonic-sonic flows with general conservative forces]
	{On subsonic and subsonic-sonic flows with general conservatives force in exterior domains}
	
	\author{Xumin Gu}
	\address{X. Gu, Department of Mathematics, Shanghai University of Finance and Economics,
		Shanghai, 200433, P. R. China; The Institute of Mathematical Sciences, The Chinese University of Hong Kong,
		Shatin, N.T., Hong Kong}
	\email{xmgu@ims.cuhk.edu.hk; xumingu11@fudan.edu.cn}

	\author{Tian-Yi Wang$^*$}\thanks{$*$ Corresponding author.}
	\address{T.-Y. Wang, Department of Mathematics, School of Science, Wuhan University of Technology,
		Wuhan, Hubei 430070, P. R. China;
		The Institute of Mathematical Sciences, The Chinese University of Hong Kong,
		Shatin, N.T., Hong Kong}
	\email{tianyiwang@whut.edu.cn; wangtianyi@amss.ac.cn}
\begin{abstract}
In this paper, we study irrotational subsonic and subsonic-sonic flows with general conservative forces in the exterior domains. The conservative forces indicate the new Bernoulli law naturally. For the subsonic case, we introduce a modified cut-off system depending on the conservative forces which needs the varied Bers skill, and construct the solution by the new variational formula. Moreover, comparing with previous results, our result extends the pressure-density relation to the general case. Afterwards we obtain the subsonic-sonic limit solution by taking the extract subsonic solutions as the approximate sequences.
\end{abstract}

\keywords{Steady flow, homentropic, irrotation, subsonic flow, subsonic-sonic limit}
\subjclass[2010]{
	35Q31; 
	35L65; 
	76N10; 
	76G25; 
	35D30
}

\maketitle

\section{Introduction}

Here we are considering the steady homentropic Euler equations with extract forces, which are written as:
\begin{eqnarray}\label{OrE}
\begin{cases}
	\mbox{div}(\rho u)=0,\\
	\mbox{div}(\rho u\otimes u)+\nabla p=\rho F,
\end{cases}
\end{eqnarray}
where $x=(x_1, \cdots, x_n)\in \R^n, n\ge 3$. $u=(u_1,\cdots,u_n)\in \R^n$ is the fluid velocity, while $\rho$, $p$, and $F$ represent the density, pressure, and extra forces respectively. For the hometropic flow, the pressure $p$ is a function of the density $\rho$, which is written as:
$
p=p(\rho).
$
As usual, we require
\begin{equation}\label{pcondition}
p'(\rho)>0,  \quad 2p'(\rho) + \rho p''(\rho) > 0 \qquad \mbox{for $\rho>0$},
\end{equation}
which include the $\gamma$-laws flow with  $p=\kappa \rho^\gamma$, for $\gamma>1$ and $\kappa>0$,
and the isothermal flows with $p=\kappa \rho$; see \cite{Courant-Friedrichs}.
The Mach number is a non-dimensional ratio of the fluid velocity to local sound
speed,
$$M=\frac{|u|}{c},$$
where
$$c=\sqrt{p'(\rho)},$$
is the local sound speed and
$$
|u|:=\Big(\sum_{i=1}^n u_i^2\Big)^{1/2}
$$
is the flow speed. The flow is subsonic when $M<1$, while the $M=1$ means the flow is locally sonic. Otherwise, $M>1$ implies flow is supersonic.

Through this paper, we consider that the extra force $F$ is conservative. 
This is reasonable since this type of forces is quite natural and important in the reality. For instance, by Newton's law of universal gravitation, the gravity field is a conservative field. Another usual example is the electric field.

Due the infinity state and the structure of the extra forces, we assume the flow is irrotational, which means the vorticity of the flow velocity
$$
\mbox{curl} u =0.
$$


One of classical problems on the steady compressible flows is the exterior domain problem.
 Let $\Gamma$ be one closed $n-1$
 dimensional hyper surfaces in $n$-dimensional Euclidean space $\R^n$ which is filled with a compressible fluid in the exterior region $\Omega$. We shall always assume $\Gamma\in
 C^{2, \alpha}$ and $\overline{\Omega}$ does not contain origin.  At the $\Gamma$ boundary, the flow satisfies the slip condition:
\begin{equation*}
u \cdot \nu =0 \qquad \mbox{on} ~\Gamma,
\end{equation*}
where  $\nu$ is the unit outward normal to the region $\Omega$. For the infinity state of the flow, after the normalization and Galilean transformation, one can assume $\lim_{|x|\rightarrow\infty} \rho(x)=1$, and $\lim_{|x|\rightarrow\infty} u=(q_\infty, 0, \cdots, 0)$. The problem is also called as airfoil problem when $F\equiv0$.

 The study of the subsonic flows is important due to its physical background and has a long research history. The first theoretical result was obtained by Frankl and Keldysh in \cite{Frankl}. They studied the subsonic flows around a two dimensional airfoil and proved the existence and the uniqueness for small data by the method of successive approximations. Later on, Bers \cite{Bers1, Bers2} proved the existence of subsonic flows with arbitrarily high local subsonic speed for the Chaplygin gas (minimal surface). By a variational method, Shiffman \cite{Shiffman1, Shiffman2} proved that, if the infinite free stream flow speed $q_{\infty}$ is less than some critical speed, there exists a unique subsonic potential flow around a given profile with finite energy. Shortly afterwards, Bers \cite{Bers3} improved  Shiffman's uniqueness results. Finn and Gilbarg \cite{Finn1} proved the uniqueness of the two dimensional potential subsonic flow around a bounded obstacle with given circulation and velocity at infinity. All the above results are related to two dimensional problems. For three (or higher) dimensional cases, Finn and Gilbarg \cite{Gilbarg1} proved the existence, uniqueness and the asymptotic behavior with implicit restrictions on Mach number $M$. Payne and Weinberger \cite{Payne} improved their results soon after. Later, Dong \cite{Dong1} extended the results of Finn and Gilbarg
 \cite{Gilbarg1} to any Mach number $M < 1$ and to arbitrary dimensions. Proceeding further,
 in \cite{Dong2}, Dong and Ou extended the results of Shiffman to higher dimensions by
 the direct method of calculus of variations and the standard Hilbert space method for the $\gamma$-law case and isothermal case. The respective incompressible case is considered in Ou \cite{ou1} and Ou-Lu \cite{ou2}.  For the rotation flow, the symmetric body case is considered recently in \cite{Chen-Du-Xie-Xin}. The another case of subsonic flow is the infinitely long nozzle case, the reader can refer to  \cite{Du-Yan-Xin, Liu-Yian, Xin1, Xin2} for results and details.

On the other hand, the existence of subsonic-sonic flows could be generated by the subsonic-sonic limit from the existed exact subsonic solutions. The first compactness framework on sonic-subsonic irrotational flows in two dimension was due to  \cite{Chen6} and \cite{Xin1} independently. The general compactness framework was introduced in \cite{Chen6} by Chen, Dafermos, Slemrod and Wang. While for the infinitely long nozzle problem,  Xie and Xin \cite{Xin1} investigated the subsonic-sonic limit of the two-dimensional irrotational flows. Later, in \cite{Xin2}, they extended the result to the three-dimensional axisymmetric flow through an axisymmetric nozzle. The compactness framework in the general multidimensional irrotational case was established in \cite{Huang-Wang-Wang}. The non-homentropic and rotation flows case is concluded by Chen, Huang, and Wang in \cite{Chen-Huang-Wang}.

We will discuss both the subsonic case and subsonic-sonic case in this paper. The general conservative forces lead a new Bernoulli law, which can not be handled by the existed process directly. For the subsonic case, we need to introduce the modified cut-off system and variation formula combining with the varied  Bers skill. Also, comparing with the previous results, we extend the pressure-density relation to the general cases, which includes  $\gamma$-laws flows and isothermal flows, basing on the delicate analysis on the phase plane. For the subsonic-sonic case, taking the extract subsonic as the approximate sequence, one can obtain the subsonic-sonic limit solution by employing the convergence theorem in \cite{Chen-Huang-Wang}.

The rest of this paper is organized as follows.
In Section 2, we establish the formulation of the problem and state the main theorem. We clarify the mathematical setting and introduce the cut-off by modifying density function in Section 3. For the modified problem, the variation formulation is used to constructing the solution in Section 4. In Section 5, the higher regularity of the modified flows is proved. Finally, in Section 6, we complete the proof by the varied Bers skill and subsonic-sonic compactness.

\section{The formulation of the problem and the main result}
Due to $F$ is conservative force, we could introduce the potential function $\psi$ such that
\begin{equation*}
F_i=\partial_{i} \psi
\end{equation*}
for $i=1,\cdots, n$. From equation $(\ref{OrE})$, we can have. 
$$
\sum_{j=1}^n \rho u_j\partial_i u_j + \partial_ip(\rho) = \rho \partial_i\psi
$$
due to the irrotational condition $\Curl u=0$.
Dividing by $\rho$ and defining  $h(\rho)$ as
$$
h(\rho) = \int_1^\rho\frac{p'(\tau)}{\tau} d \tau,
$$
we get
$$
\nabla\left[\frac{1}{2}|u|^2 + h(\rho) -\psi\right]=0.
$$
Then the Bernoulli law comes to
\begin{equation}\label{Brelation}
\frac{1}{2}|u|^2 + h(\rho) =\psi,
\end{equation}
with modifying a constant.
Without loss of generality, we assume $\psi$ is bounded and
\begin{equation}\label{psicondition}
\lim_{\rho\rightarrow+\infty}h(\rho)>\psi>\lim_{\rho\rightarrow0^+}H(\rho),\end{equation}
where
$$
H(\rho)=\left(\frac{p'(\rho)}{2}+h(\rho)\right).
$$
From \eqref{pcondition} and \eqref{psicondition}, it is easy to see $h(\rho)$ has the respective inverse function $h^{-1}$, which leads  the presentation of density:
\begin{equation}\label{equ-rho}
\rho=h^{-1} \left(\psi-\frac{|u|^2}{2}\right)
\end{equation}
which is equivalence to \eqref{Brelation}. Then, Mach number can be regarded as the function of $u$ and $\psi$, which is written as $M(u;\psi)$.

Within this paper, we will consider the following problem:

\textbf{Problem 1 ($q_\infty$)}:  Find functions
$u=(u_1,\cdots,u_n)$ satisfy
\begin{eqnarray}\label{p1equation}
\begin{cases}
\mbox{div}\left(\rho u\right)=0,\\
\mbox{curl} u =0,
\end{cases}
\end{eqnarray}
 with the
Bernoulli law \eqref{Brelation} in $\Omega$.
And the slip boundary
condition
\begin{equation}\label{3.1}
(\rho u)\cdot \nu=0\ \ \mbox{on}\ \ \Gamma,
\end{equation}
where $\nu$ denotes the unit inward normal of domain
$\Omega$, and the limit
\begin{equation*}
\lim_{|x|\rightarrow\infty} u(x)=(q_\infty, 0, \cdots, 0)
\end{equation*}
exists and is finite.

\begin{remark}
If the flow without vacuum, which means $\inf_{x\in\Gamma}\rho(x)>0$, \eqref{3.1} can be written as
\begin{equation}\label{3.1+}
u \cdot\nu=0\ \ \mbox{on}\ \ \Gamma.
\end{equation}
\end{remark}

Our main result is the following theorem:
\begin{theorem}\label{mainthm}
For the given $\psi$ satisfies \eqref{psicondition} and
\begin{equation}\label{psicondition1}
\partial_1\psi\in  L^{\frac{2n}{n+2}}(\Omega)~ ~ \mbox{ and } ~ ~ |x|^\beta\nabla\psi\in L^q(\Omega) ~ \mbox{ for } ~ q>n , ~~ \beta>1-\frac{n}{q}.
\end{equation}
(1) There exists a positive number $\hat{q}$, if $q_\infty<\hat{q}$, then there exists  an unique solution $u \in C^{1, \alpha_0}$ for some $0<\alpha_0< 1$ of \textbf{Problem 1 $(q_\infty)$}, and Mach number $M(u; \psi)<1$.

(2) Let $q_{\infty}^\varepsilon\rightarrow \hat{q}$ as $\varepsilon\rightarrow0$, with
$q_\infty^{\varepsilon}<\hat{q}$. And
$u^{\varepsilon}=(u_1^{\varepsilon},\cdots,u_n^{\varepsilon})$ be
the corresponding solutions to  \textbf{Problem 1 $(q_\infty^\varepsilon)$}. Then, as
$q_{\infty}^\varepsilon\rightarrow \hat{q}$, the solution
sequence $u^{\varepsilon}(x)$  possess a subsequence (still denoted
by) converge a.e. in $\Omega$ to $\bar{u}(x)=(\bar{u}_1, \cdots, \bar{u}_n)(x)$ which is a weak solution of \textbf{Problem 1 ($\hat{q}$)}.
Furthermore, $\bar{u}$ and $\bar{\rho}$, which is defined through \eqref{Brelation}, also satisfies
$(\ref{OrE})_2$ in the sense of
distributions and the
boundary condition \eqref{3.1} as the normal trace of the
divergence-measure field on the boundary (see
\cite{Chen7}).
\end{theorem}

\begin{remark}
It is easy to check for $n=3$ the gravity generated by the solid domain $\Omega^c$ :
$$
\psi=\int_{\Omega^c}\frac{\rho_s(y)}{|x-y|} d y
$$
satisfies the conditions \eqref{psicondition} and \eqref{psicondition1} on $\psi$, where $\rho_s\in L^1(\Omega^c)$ present the density distribution of  $\Omega^c$ with the finite mass. It also can be applied to the electric field.
\end{remark}

\begin{remark}
	It is noticeable that when $\beta>\frac{n}{2}+1-\frac{n}{q}$ with $q>n$, $|x|^\beta\nabla\psi\in L^q(\Omega)$ includes the both sub-conditions in \eqref{psicondition1} by the standard H\"{o}lder's inequality.
\end{remark}

\begin{remark}
In part (1) of the Theorem \ref{mainthm},  the regularity of  $u$ are limited by $\psi$. One can improve the  regularity of $u$ and $\rho$ by imposing the further smooth condition on $\psi$.
\end{remark}

\section{Mathematical setting and Modification of the density function}
In this section, we will transfer \textbf{Problem 1 ($q_\infty$)} to a second order partial differential problem, and introduce a respective subsonic cut-off.

For the irrotation equation $(\ref{p1equation})_2$, we could introduce the flow potential $\phi$, which satisfies:
\begin{equation*}
u=\nabla \phi.
\end{equation*}
Then, the slip condition \eqref{3.1+} on the boundary $\Gamma$ comes to
\begin{equation*}
\quad \frac{\partial\phi}{\partial \nu}=0.
\end{equation*}
We also give the infinity condition that
\begin{equation*}
\lim_{|x|\rightarrow\infty}\nabla\phi(x) = (q_\infty, 0, \cdots, 0).
\end{equation*}

Then, the presentation of density \eqref{equ-rho} comes to:
\begin{equation*}
\rho(\left|\nabla\phi\right|^2-2\psi):=h^{-1} \left(\psi-\frac{|u|^2}{2}\right).
\end{equation*}
Then, we come to the second order equation form $(\ref{p1equation})_1$:
$$
\mbox{div}\left(\rho(\left|\nabla\phi\right|^2-2\psi)\nabla\phi\right)=0.
$$
Then, \textbf{Problem 1 ($q_\infty$)} comes to:

\textbf{Problem 2 ($q_\infty$)}:  Find $\phi(x)$ such that
\begin{eqnarray}\label{equ-bc-ic}
\begin{cases}
\mbox{div}\left(\rho(\left|\nabla\phi\right|^2-2\psi)\nabla\phi\right)=0, \quad & \mbox{in} ~\Omega,\\
\frac{\partial \phi}{\partial \nu} =0, \quad & \mbox{on} ~\Gamma,\\
\lim_{|x|\rightarrow\infty}\nabla\phi(x) = (q_\infty, 0, \cdots, 0).
\end{cases}
\end{eqnarray}

From the direct calculating, $(\ref{equ-bc-ic})_1$ comes to:
$$
\sum_{i=1}^n \partial_i\left(\rho \left(|\nabla\phi|^2-2\psi\right) \partial_i \phi\right)
= \sum_{i, j=1}^n a_{ij} \partial_{ij} \phi+\sum_{i=1}^n b_i \partial_i \phi=0
.
$$
where
$$
a_{ij} = \rho \delta_{i,j}- \rho'(\left|\nabla\phi\right|^2-2\psi)\partial_i\phi \partial_j \phi = \rho\left(\delta_{i,j}-\frac{\partial_i\phi \partial_j \phi}{c^2}   \right),
$$
and
$$
b_i=\frac{\rho \partial_i\psi}{c^2}.
$$
Then, for  $\xi\in\mathbb{R}^n$,
$$
\rho(1-M^2)|\xi|^2\leq a_{ij}\xi_i\xi_j\leq \rho |\xi|^2.
$$
Then, we could see $(\ref{equ-bc-ic})_1$ is elliptic if and only if the flow is subsonic, and it will degenerate in the subsonic-sonic case.  It is noticeable that without \textit{a prior} estimate on $|\nabla\phi|$, the potential equation  $(\ref{equ-bc-ic})_1$ is not guaranteed to be uniform ellipticity. Therefore, we need to introduce the following cut-off.

From \eqref{pcondition}, $H(\rho)$ is an increase function respect to $\rho$. For fixed $\psi$, $$q_{cr}(\psi):=\left(2\psi-2h\left(H^{-1}(\psi)\right)\right)^{\frac{1}{2}}$$
 is the critical speed. From the direct calculation, one can show the flow is subsonic ($M<1$) if and only if $|u|<q_{cr}(\psi)$.

Now, we introduce a modified problem of \textbf{Problem 2 ($q_\infty$)}, which is uniformly elliptic by
presenting a way to modify the density $\rho$. For any small $\theta>0$, we  define $\tilde{\rho}$ as
\begin{equation*}
\tilde{\rho}(|\nabla\phi|^2, \psi):=\begin{cases}
\rho\left(|\nabla\phi|^2-2\psi\right), \quad &
\mbox{if} ~|\nabla\phi|\leq(1-2\theta)q_{cr}(\psi),\\
\mbox{monotone ~smooth ~connection},\quad &
\mbox{otherwise},\\
\sup_{x\in\Omega}\left\{\rho\left((1-\theta)^2q^2_{cr}(\psi(x))-2\psi(x)\right)\right\},\quad &
\mbox{if} ~|\nabla\phi|\geq(1-\theta)q_{cr}(\psi).
\end{cases}
\end{equation*}
And, we denote $\tilde{\rho}_v(v, w):=\frac{\partial}{\partial v}\tilde{\rho}(v, w)$, and $\tilde{\rho}_w(v, w):=\frac{\partial}{\partial w}\tilde{\rho}(v, w)$.

Then, \textbf{Problem 3 ($q_\infty$)} is defined as:
 Find $\phi(x)$ such that:
\begin{eqnarray}\label{equ-modify}
\begin{cases}
\mbox{div}\left(\tilde{\rho}\left(|\nabla\phi|^2, \psi\right) \nabla \phi\right)=0, \quad & \mbox{in} ~\Omega,\\
\frac{\partial\phi}{\partial \nu}=0, \quad & \mbox{on} ~\Gamma,\\
\lim_{|x|\rightarrow\infty}\nabla\phi(x) = (q_\infty, 0, \cdots, 0).
\end{cases}
\end{eqnarray}
After the similar calculation with  \textbf{Problem 2 ($q_\infty$)}, $\psi$ satisfies:
\begin{equation}\label{equ-modify1}
\sum_{i, j=1}^n \tilde{a}_{ij} \partial_{ij} \phi+\sum_{i=1}^n \tilde{b}_i \partial_i\phi=0
.
\end{equation}
where
$$
\tilde{a}_{ij} = \tilde{\rho} \delta_{i,j}- 2\tilde{\rho}_v\partial_i\phi \partial_j \phi,
$$
and
$$
\tilde{b}_i= \tilde{\rho}_w \partial_i\psi.
$$
And, for $\xi\in\mathbb{R}^n$,
\begin{equation} \label{cutoffellpitic}
 C|\xi|^2\leq \tilde{a}_{ij}\xi_i\xi_j\leq C^{-1} |\xi|^2,
\end{equation}
and
$$
|\tilde{b}_i \partial_i\phi|\leq C|\partial_{i}\psi|,
$$
where $C$ is a positive number dependent on $\theta$ and $\psi$.

\section{A variation Formulation}
In this section, we solve \textbf{Problem 2 ($q_\infty$)} by a variational method. To do that,  we need a suitable Hilbert space. Due to \cite{Dong1,Dong2,ou1,ou2}, the suitable function space is from the following:

\begin{theorem}
	Define a function set composed of all the function on $\Omega$ which are really the restrictions of some $C_{0}^{\infty}(\mathbb{R}^n)$ function on $\Omega$:
	\begin{equation*}
		\mathcal{V}_0 = \left\{\varphi(x),~ x\in\Omega ~ |~  \varphi(x)=\tilde{\varphi}(x) \mbox{~for ~some} ~\tilde{\varphi}\in C_0^{\infty} (\R^n)
		\right\}.
	\end{equation*}
	Then under the norm
	\begin{equation*}
		\|\varphi\|_{\mathcal{V}} = \left(\int_\Omega |\nabla \varphi|^2dx\right)^{\frac{1}{2}},
	\end{equation*}
	$\mathcal{V}_0$ expands a Hilbert space $\mathcal{V}$ if $n\geq 3$.
\end{theorem}

\begin{theorem}
There exists a constant $C_{n}(\Omega)$, such that for $\varphi\in\mathcal{V}$,
\begin{equation}\label{pineq}
\left(\int_{\Omega}|\varphi|^{\frac{2n}{n-2}}dx\right) ^{\frac{n-2}{2n}}\leq C_{n}(\Omega) \left(\int_{\Omega}|\nabla \varphi|^{2}dx\right)^{\frac{1}{2}}.
\end{equation}
\end{theorem}
Now we propose our variational problem for \eqref{equ-modify} in the space $\mathcal{V}$. For the given $\psi$, let
$$
\phi = \varphi +  q_{\infty} x_1, \quad
G(\Lambda, \psi)=\frac{1}{2}\int_0^\Lambda\tilde{\rho}(v, \psi)dv,
$$
and we define a functional $I(\varphi,  q^{\infty})$
\begin{align}\label{equ-I}
		I(\varphi,  q_{\infty})=&\int_{\Omega}\left[G\left(|\nabla\phi|^2, \psi\right)
		-G\left(( q_{\infty})^2, \psi\right)-2G_v\left(( q_{\infty})^2, \psi\right) q_{\infty}
		\left(\partial_1\phi- q_{\infty}\right)\right]dx\\
		&+\int_{\Omega}2G_{vw}\left(( q_{\infty})^2, \psi\right) q_{\infty}\partial_1\psi\left(\phi- q_{\infty} x_1\right)\,dx-\int_\Gamma 2 G_v \left(( q_{\infty})^2, \psi\right) q_{\infty}
		\left(\phi- q_{\infty} x_1\right)\nu_1ds\nonumber,
\end{align}
where $\nu_1$ is the first component of the outward normal $\nu=(\nu_1, \cdots, \nu_n)$.

With the condition $\partial_1 \psi \in L^{\frac{2n}{n+2}}(\Omega)$, the existence of a solution to \textbf{Problem 3 ($q_\infty$)} is equivalent to the following variational problem:

\textbf{Problem 4 ($q_\infty$)}: Find a minimizer $\bar\varphi\in\mathcal{V}$  such that
$$
I\left(\bar\varphi, q_{\infty}\right)=\min_{\varphi\in\mathcal{V}}I\left(\varphi, q_{\infty}\right).
$$

It is direct to check that the equation \eqref{equ-modify} is the Euler-Lagrangian equation of our variation problem. Let
$\eta\in C_0^\infty({\R}^n)$,
the first variation of $I\left(\varphi,  q^{\infty}\right)$ with
$\eta$ would be
\begin{align*}
		0=&\int_\Omega\left[2G_v \left(|\nabla\phi|^2, \psi \right)\nabla\phi\nabla\eta-2G_v
		\left(( q_{\infty})^2, \psi\right) q_{\infty}\partial_1\eta\right]dx\\
		&+\int_\Omega 2 G_{vw}\left(( q_{\infty})^2, \psi\right) q_{\infty}\partial_1\psi\eta\,dx-\int_\Gamma 2 G_v \left(( q_{\infty})^2, \psi \right) q_{\infty}\eta\nu_1dS.
\end{align*}
The last three terms are cancelled by integration by part. Then, the first variation of
$I\left(\varphi,  q_{\infty}\right)$
associated with $\eta$ is
\begin{align*}
		0=&\int_\Omega 2G_v \left(|\nabla\phi|^2, \psi\right)\nabla\phi\nabla\eta dx\\
		=&\int_\Omega \tilde{\rho} \left(|\nabla\phi|^2, \psi\right)\nabla\phi\nabla\eta dx\\
		=&-\int_\Omega\mbox{div} \left(\tilde{\rho} \left(|\nabla\phi|^2, \psi\right)
		\nabla\phi\right)\eta dx
		-\int_\Gamma \tilde{\rho}\left(|\nabla\phi|^2, \psi\right)
		\frac{\partial\phi}{\partial\nu}\eta dS.
\end{align*}

For our variational problem, we have the following theorem:
\begin{theorem}\label{th3} \textbf{Problem 4 ($q_\infty$)} has a unique minimizer $\bar\varphi\in\mathcal{V}$. Moreover,
	\begin{equation}
		\label{ff}
		\int_{\Omega}|\nabla\bar\varphi|^2\, dx \leq C,
	\end{equation}
where $C$ dependants on $q_\infty$, $\psi$ and $\Omega$.
\end{theorem}
\begin{proof}
	Step 1. $I(\varphi, q_{\infty})$ is coercive in $\mathcal{V}$.
	
	Firstly, we denote
	$$
	B\left(\varphi,  q_{\infty}\right)=\int_\Omega \left[G\left(|\nabla\phi|^2, \psi\right)
	-G\left(\left( q_{\infty}\right)^2, \psi\right)-2G_v\left(\left( q_{\infty}\right)^2, \psi\right)
	q_{\infty}\left(\partial_1\phi- q_{\infty}\right)
	\right]dx,
	$$
	and will prove $B\left(\varphi,  q_{\infty}\right)$ is uniformly convex in the space $\mathcal{V}$.
	
	Let $e_1=(1, 0, \cdots ,0)$, then we have
	$$
	\nabla\phi -  q_{\infty} e_1 =\nabla\varphi.
	$$
	We denote $p=(p_1, \cdots, p_n)$, $F(p)=G\left(|p|^2, \psi\right)$. Then by direct computation, we can get that
	\begin{align*}
			&G\left(|\nabla\phi|^2, \psi\right)-G\left(\left( q_{\infty}\right)^2, \psi\right)
			-2G_v \left(\left( q_{\infty}\right)^2, \psi\right) q_{\infty}
			\left(\partial_1\phi- q_{\infty}\right)\\
			&= F(\nabla\phi) - F\left( q_{\infty} e_1\right)
			-\partial_{p_1}F\left( q_{\infty} e_1\right)\left(\nabla\phi- q_{\infty} e_1\right)\\
			&=\sum_{i, j=1}^n\int_0^1(1-t)\partial_{p_ip_j} F\left(t\nabla\phi+(1-t) q_{\infty} e_1\right)dt
			\partial_i\varphi\partial_j\varphi.
	\end{align*}
	It is easy to check $\partial_{pp}^2F$ is uniformly positive.  In fact, we have
	$$
	\left(\partial_{pp}^2F(p)\right)_{i,j} =\tilde {\rho}\delta_{ij}
	+ 2 \tilde{\rho}_v p_ip_j=\tilde{a}_{ij}.
	$$
	From the cut off $\tilde{\rho}$ property \eqref{cutoffellpitic}, we get the uniformly positivity of $\partial_{pp}^2F$. As consequence,
	\begin{equation}\label{ineq-phi}
		\frac{C_1}{2}\int|\nabla\varphi|^2dx\leq B\left(\varphi,  q_{\infty}\right)\leq
		\frac{\tilde{C}_1}{2}\int|\nabla\varphi|^2dx.
	\end{equation}
	
	With a similar produce, for any $\varphi_1, \varphi_2\in \mathcal{V}$, we have that
	\begin{align*}
			&B\left(\varphi_1,  q_{\infty}\right) + B \left(\varphi_2,  q_{\infty}\right)
			-2B\left(\frac{\varphi_1+\varphi_2}{2},  q_{\infty}\right)\\
			&=F(\nabla\phi_1) + F(\nabla\phi_2)-2F\left(\frac{\nabla\phi_1+\nabla\phi_2}{2}\right)\\
			&\geq \frac{C_1}{2} \|\varphi_1-\varphi_2\|_{\mathcal{V}}^2,
	\end{align*}
	which proves the uniformly convexity of $B$.
	
	Secondly, the surface integrand in \eqref{equ-I} is continuous linear functional of $\varphi$. In fact, by Hardy's inequality, there is a constant $C(\Omega)$ such that
	\begin{equation}\label{ineq-phi-surface}
		\int_\Gamma |\varphi|dS \leq C(\Omega) \left(\int_\Omega |\nabla\varphi|^2\, dx\right)^{\frac{1}{2}}.
	\end{equation}
	
	Thirdly, the body integrand
	$$\int_{\Omega}2G_{vw}\left(( q_{\infty})^2, \psi\right) q_{\infty}\partial_1\psi\left(\phi- q_{\infty} x_1\right) dx
	=\int_{\Omega}2G_{vw}\left(( q_{\infty})^2, \psi\right) q_{\infty}\partial_1\psi\varphi dx$$
	 is also continuous linear functional of $\varphi$, which relies on $G_{vw}\left(( q_{\infty})^2, \psi\right)\in L^{\infty}(\Omega)$ and $\partial_{1}\psi \in L^{\frac{2n}{n+2}}(\Omega)$ and the Poincare's inequality \eqref{pineq} on $\varphi$.

	Now we can prove the coercive for $I(\varphi, q_{\infty})$. By \eqref{ineq-phi-surface}, we have
	\begin{align*}
			&\left|\int_\Gamma 2 G_v\left(\left( q_{\infty}\right)^2, \psi\right)
			q_{\infty} \left(\phi- q_{\infty} x_1\right)\nu_1 dS\right|\\
			&=\left|\int_\Gamma 2 G_v\left(\left( q_{\infty} \right)^2, \psi\right) q_{\infty} \varphi \nu_1 dS\right|\\
			&\leq C\int_\Gamma |\varphi|dS \\
			&\leq\frac{C_1}{8} \int_\Omega |\nabla\varphi|^2 dx +C_2.
	\end{align*}
	Similarly, by \eqref{pineq}, we have
	\begin{equation*}
		|\int_{\Omega}2G_{vw}\left(( q_{\infty})^2, \psi\right) q_{\infty}\partial_1\psi\left(\phi- q^{\infty} x_1\right)\,dx| \leq\frac{C_1}{8} \int_\Omega |\nabla\varphi|^2 dx +C_3.
	\end{equation*}
	Therefore from \eqref{ineq-phi}, we get
	\begin{equation}\label{I-phi}
		I\left(\varphi,  q_{\infty}\right) \geq \frac{C_1}{4}\int_\Omega |\nabla\varphi|^2 dx -C_4.
	\end{equation}
	
	Step 2. The existence of minimizer $\bar\varphi \in\mathcal{V}$. 	
	
	First, we examine the continuity of $I(\varphi,  q_{\infty})$ in $\mathcal{V}\times \R^+$. For the surface integral part of $I(\varphi,  q_{\infty})$ and the body integrand $\int_{\Omega}2G_{vw}\left(( q_{\infty})^2, \psi\right) q_{\infty}\partial_1\psi\left(\phi- q_{\infty} x_1\right)\,dx$, the continuity is showed in Step 1.
	
	For $B(\varphi,q_{\infty})$, let $\phi_i=\varphi_i+ q_{\infty} x_1$, for $i=1,2$. Recall the definition of $F$ in Step 1, we have
	\begin{align*}
			&\left[G\left(|\nabla\phi_1|^2, \psi\right)-G\left(\left( q_{\infty}\right)^2, \psi\right)
			-2G_v\left(\left( q_{\infty}\right)^2, \psi\right) q_{\infty}
			\left(\partial_1\phi_1- q_{\infty}\right)\right]
			\\&-\left[G\left(|\nabla\phi_2|^2, \psi\right)-G\left(\left( q_{\infty}\right)^2, \psi\right)
			-2G_v\left(\left( q_{\infty}\right)^2, \psi\right) q_{\infty}
			\left(\partial_1\phi_2- q_{\infty}\right)\right]\\
			&=\left[F(\nabla\phi_1) - F\left( q_{\infty} e_1\right)
			-\partial_{p} F \left( q_{\infty} e_1\right) \nabla\varphi_1\right]
			-\left[F(\nabla\phi_2) - F\left( q_{\infty} e_1\right)
			-\partial_{p} F \left( q_{\infty} e_1\right) \nabla\varphi_2\right]\\
			&=\int_0^1\partial_pF\left(t\nabla\varphi_1+(1-t)\nabla\varphi_2
			+ q_{\infty} e_1\right)dt (\nabla\varphi_1-\nabla\varphi_2)
			-\partial_pF\left( q_{\infty} e_1\right)(\nabla\varphi_1-\nabla\varphi_2)\\
			&=\sum_{i, j=1}^n \int_0^1\int_0^1 \partial_{p_ip_j} F \left(st\nabla\varphi_1+s(1-t)\nabla\varphi_2
			+ q_{\infty} e_1\right)ds \left[t\partial_i\varphi_1+(1-t)\partial_i\varphi_2\right]dt
			(\partial_j\varphi_1-\partial_j\varphi_2).
		\end{align*}
	Then combining with the surface integral, we have
	$$|I(\varphi_1, q_{\infty})-I(\varphi_1,q_{\infty})|\leq C(1+\|\varphi_1\|_{\mathcal{V}}+\|\varphi_2\|_{\mathcal{V}})\|\varphi_1-\varphi_2\|_{\mathcal{V}},$$
	where $C$ depends on $q_{\infty}$, $n$, $\Omega$, $\theta$, and $\psi$.
	
	The continuity of $I\left(\varphi,  q_{\infty}\right)$ on $ q_{\infty}$ follows from the equality below,
	\begin{align}\label{4.8}
	I(\varphi,  q_{\infty})=&\int_\Omega \int_0^1 (1-t) D_{p_ip_j}^2
	F\left(t\nabla\varphi+ q_{\infty} e_1\right)dt \partial_i\varphi \partial_j\varphi dx\\
	&+\int_{\Omega}2G_{vw}\left(( q_{\infty})^2, \psi\right) q_{\infty}\partial_1\psi \varphi\,dx-\int_\Gamma 2 G_v \left(( q_{\infty})^2, \psi\right) q_{\infty}
	\varphi \nu_1ds\nonumber.
	\end{align}
	Then, by applying the standard Hilbert method, we know every minimizing sequence $\phi_m=\varphi_m +  q_{\infty} x_1$ is convergent. Then the continuity of the functional with respect to $\varphi$ in $\mathcal{V}$ will guarantee the existence of a minimizer $\bar\varphi$.

	Step 3. The uniqueness of minimizer $\bar\varphi$.
	
	We pick a minimizing sequence composed of two minimizers alternatively. A minimizing sequence is always convergent, so any two minimizers are the same.
	
	Step 4. To prove \eqref{ff}, we compare $I(\varphi, q_{\infty})$ with $I(0, q_{\infty})$. Then \eqref{ff} follows by \eqref{I-phi} easily.
\end{proof}

\section{modified flows}
In the last section, we have constructed the unique solution of \textbf{Problem 3 ($q_\infty$)} by solving \textbf{Problem 4 ($q_\infty$)}. In this section, we will show the further regularity.

First, we need the following proposition:
\begin{proposition}\label{prop-modify}
	Let $a^l_{ij}$ for $i,j = 1, \dots, n$ be measurable functions on $B_1$, and $\lambda$ be a positive constant. Assume that
	$$
	\forall ~\xi\in\R^n, ~\lambda|\xi|^2\leq a^l_{ij}\xi_i\xi_j\leq \lambda^{-1} |\xi|^2,
	~\mbox{and} ~f^l_i\in L^q, ~ q >n.
	$$
	Let $w(y)$ be a function in $H^1$, suppose
	$$
	\sum_{i,j=1}^n\partial_i\left[a^l_{ij}(y) \partial_jw(y)\right] + \sum_{i=1}^n\partial_if^l_i =0
	$$
	is satisfied weakly. Then $w(y)$ is H\"{o}lder continuous in $B_{{1}/{2}}$ and there exist two constants $0<\alpha\leq 1$, $k$, depending on $\lambda$ such that
	$$
	\sup_{y\in B_{1/2}}|w(y)|\leq k\left(||w||_{L^2(B_1)}+ ||f_i^l||_{L^q(B_1)} \right),
	$$
	$$
	\sup_{y_1, y_2\in B_{1/2}}\frac{|w(y_1)-w(y_2)|}{|y_1-y_2|^{\alpha}}
	\leq k\left(||w||_{L^2(B_1)}+ ||f_i^l||_{L^q(B_1)} \right).
	$$
\end{proposition}
The proof of this proposition can be found in \cite{Gilbarg-Trudinger}.
\begin{lemma}\label{lem-3}
	Let $\Omega'$ be a bounded interior subregion of $\Omega$, for $\nabla\psi\in L^q$,
	$q>n$, then there are constants $0<\alpha<1$ and $C$ depending on $\Omega$, $\Omega'$,  $ q_{\infty}$ and $\psi$ such that
	$$
	\sup_{x\in\Omega'}|\nabla\phi|\leq C,
	$$
	$$
	\sup_{x_1,x_2\in\Omega'} \frac{|\nabla\phi(x_1)-\nabla\phi(x_2)|}{|x_1-x_2|^\alpha}\leq C.
	$$
\end{lemma}
\begin{proof}
	Denote $\varphi'=\partial_k\phi$ for $k=1, \cdots, n$. Take the $k$-th partial derivative of the equation $(\ref{equ-modify})_1$ formally to get that
	$$
	\sum_{i, j= 1 }^n \partial_i \left( \tilde{a}_{ij} \partial_j \varphi' \right)+ \sum_{i=1}^n \partial_i\left(\tilde{\rho}_w\partial_k\psi\partial_i\phi\right)=0.
	$$
	By the definition of the cut-off density $\tilde{\rho}$, $\tilde{a}_{ij}$ has uniformly positive eigenvalues, the equation is hence uniform elliptic.
	Also, for $i, k=1, \cdots, n$,
	$$
	|\tilde{\rho}_w\partial_k\psi\partial_i\phi|\leq C |\partial_k\psi|
	$$
		By $\nabla \psi\in L^q$, we can show 
	$
	\tilde{\rho}_w\partial_k\psi\partial_i\phi
	$
	are bounded in $L^q$. By proposition \ref{prop-modify}, \eqref{ff}, with $B$ being scaled to arbitrary ball, the lemma follows directly although the proof has been formal. This formality can be substantiated by considering approximation of derivatives by finite differences, a standard practice in elliptic PDE theory.
\end{proof}

\begin{lemma}\label{lem-4}
	Let $\Omega'$ be $B_R\cap(\Omega\cup \Gamma)$ for a large $R$, then conclusion of Lemma \ref{lem-3} holds.
\end{lemma}
\begin{proof}
	We first remark that a smooth solution $U$ to
	\begin{eqnarray*}
		\begin{cases}
			\mbox{div}\left[\tilde\rho\left(|\nabla U|^2, \psi\right)\nabla U\right]=0, \quad &\mbox{in} ~ \Omega',\\
			U=\phi, \quad &\mbox{on} ~|x|=R,\\
			\frac{\partial U}{\partial \nu}=0, \quad &\mbox{on} ~\Gamma,
		\end{cases}
	\end{eqnarray*}
	would be the same as $\varphi$ if such a solution does exist.
	
	Indeed, both $U$ and $\varphi$ satisfy weakly Euler-Lagrangian equation if
	$$
	\delta \int_{B_k\cap \Omega}G\left(|Dw|^2, \psi\right)dx =0
	$$
	subject to the boundary constraint $w=\varphi$ on $|x|=R$.
	
	Hence, both are critical points of a uniformly strict convex functioned which has only one critical points as minimizer.
	For the existence of $U$, can be found in \cite[Chapter 10]{LU_68}, it is based on the a prior estimate of $U$, see \cite{Gilbarg-Trudinger}.
\end{proof}

\begin{lemma}\label{lem-5}
	There is the continuity estimate of $\nabla\phi$ at infinity:
	\begin{equation}\label{5.2}
	\left|\nabla\phi(x)-\left( q_{\infty}, 0,  \cdots , 0\right)\right|\leq
	\frac{C}{(1+|x|)^{\beta'}},
	\end{equation}
	where $\beta'=\min\{\frac{n}{2}, \beta+\frac{n}{q}-1\}$.
\end{lemma}
\begin{proof}
	Since $\phi=\varphi+ q_{\infty} x_1$. 
	Let
	$\phi'=\partial_k\phi$, $\varphi'=\partial_k\varphi$ for $k=1,\dots,n$. Then $\phi'$ satisfies
	$$
	\sum_{i, j =1}^n\partial_i\left(\tilde{a}_{ij}\partial_j\phi'\right)+\sum_{i=1}^n\partial_i\left(\tilde{\rho}_w
	\partial_i\phi\partial_k\psi\right)=0.
	$$
	So does $\varphi'$.
	\begin{equation}\label{5.3}
	\sum_{i, j=1}^n\partial_i\left(\tilde{a}_{ij}\partial_j\varphi'\right)+\sum_{i=1}^n\partial_i\left[\tilde{\rho}_w
	\partial_i\phi\partial_k\psi\right]=0.
	\end{equation}
	By \eqref{ff} in Theorem \ref{th3},
	\begin{equation}\label{5.4}
	\int_ \Omega|\varphi'|^2 dx\leq\int_\Omega|\nabla\varphi|^2dx\leq C,
	\end{equation}
	and \eqref{5.3} will be shown to be enough for Lemma \ref{lem-5}.
	
	For sufficiently large $R$, $\left\{\frac{R}{2}<|x|<2R \right\} \subset \Omega$. Define
	\begin{equation}\label{5.5}
	w(y) =R^{\beta'} \varphi'(Ry)
	\end{equation}
	on $\left\{\frac{1}{2}<|y|<2\right\}$. By \eqref{5.3} and \eqref{5.4}, $w(y)$ satisfies
	$$
	\sum_{i,j=1}^n\frac{\partial}{\partial y_i}\left(\tilde{a}_{ij}\frac{\partial w}{\partial y_j}\right) +\sum_{i=1}^n
	R^{\beta'+1}\frac{\partial}{\partial y_i}\left(\tilde{\rho}_w 	\partial_i\phi\partial_k\psi\right)=0, \quad
	\mbox{in} ~ \left\{\frac{1}{2}<|y|<2 \right\},
	$$
	$$
	\int_{\left\{\frac{1}{2}<|y|<2 \right\}}|w(y)|^2 dy \leq C, ~ ~ \mbox{ and } ~ ~ \int_{\left\{\frac{1}{2}<|y|<2 \right\}}|R^{\beta'+1}\nabla_x\psi(y)|^q dy \leq C,
	$$
	where the argument of $\tilde{a}_{ij}$ is $Ry$ if the $a_{ij}$ are taken as functions of $x$.  The last inequality is due to the second condition for $\psi$ in \eqref{psicondition1} and $\beta'+1\leq \beta+\frac{n}{q}$. Applying Proposition \ref{prop-modify}, we have
	$$
	|w(y)|\leq C \quad \mbox{for} ~|y|=1.
	$$
	Going back to \eqref{5.5}, we have for sufficiently large $|x|$, for $i=1, \cdots, n$,
	$$
	\left|\partial_i \varphi(x)\right|\leq \frac{C}{|x|^{\beta'}},
	$$
	which leads to \eqref{5.2}, combining with the result of Lemma \ref{lem-4}.
\end{proof}

The local $C^{2, \alpha_0}$ H\"{o}lder estimate on $\phi$ can be obtained through $(\ref{equ-modify1})$ by the standard elliptic estimate, while $\psi\in C^{\alpha'}$ since \eqref{psicondition} and \eqref{psicondition1}.


Now, we settle the modified problem \textbf{Problem 3 ($q_\infty$)}.

\section{subsonic flow and subsonic-sonic flow in space}
In this section, we will complete the proof Theorem \ref{mainthm}. The first step is to show the uniqueness of the modified flow. Then, we can release the cut-off base on the Bers skill and complete the proof of the subsonic part of main theorem. Then, we will take the subsonic-sonic limit by the compactness theorem in \cite{Chen-Huang-Wang}.

\begin{theorem}\label{6.1}
	For every $q_\infty$, there is a unique classical solution such that
	$$
	\phi=\varphi+q_\infty x_1 \quad \mbox{with} ~ \varphi\in \mathcal{V}.
	$$
	Furthermore, the velocity field $\nabla\phi$ depends on $q_\infty$ continuously and in particular $\max_\Omega|\nabla\phi|$ is a continuous function of $q_\infty$.
\end{theorem}
\begin{proof}	
	The existence follows from the existence of the variational problem in Theorem \ref{th3} and the regularity estimates in Lemma \ref{lem-3} and \ref{lem-4} and \ref{lem-5}.
	
	To prove the uniqueness, we note that two classical solutions
	$$
	\phi_i=\varphi_i+q_\infty x_1, \quad i=1, 2 \quad \mbox{with} ~\varphi_i\in \mathcal{V}
	$$
	would be both critical points of $I\left(\varphi, q_\infty\right)$ as we defined. We denote $I'_v$ is the Fr\'{e}chet derivative, then we have
	\begin{align*}
					0=&\left(I'_v\left(\varphi_1,q_\infty\right)
			-I'_v\left(\varphi_2,q_\infty\right), \varphi_1-\varphi_2\right)\\
			=&\int_\Omega \left[\partial_pF\left(\nabla\varphi_1+q_\infty e_1\right)
			-\partial_pF\left(\nabla\varphi_2+q_\infty \right)\right] (\nabla\varphi_1-\nabla\varphi_2)dx\\
			=& \int_\Omega \int_0^1 \partial_{p_ip_j}^2F\left(t\nabla\varphi_1+(1-t)\nabla\varphi_2 + q_\infty e_1\right)\partial_i(\varphi_1-\varphi_2)\partial_j(\varphi_1-\varphi_2) d t d x\\
			\geq& C \|\varphi_1-\varphi_2\|_{\mathcal{V}}^2.
			\end{align*}
	Hence $\varphi_1=\varphi_2$.
	
	Now we prove the continuous dependence of solutions on $q_\infty$.
	
	Let $q_\infty^m$ be a convergent sequence, $q_\infty^m\rightarrow \bar{q}_\infty$. Denote $\phi^m=\varphi^m+q_\infty^m x_1$, $m\in\mathbb{N}$	as the solution sequence. First we show that $\varphi_m\rightarrow \bar{\varphi}$ in $\mathcal{V}$ by using that $\phi^m$ is a minimizing sequence of $I\left(\phi, \bar{q}_\infty\right)$.
	
	In fact, since $q_\infty^m\rightarrow \bar{q}_\infty$ as $m\rightarrow\infty$, all estimates in Lemma \ref{lem-3}, \ref{lem-4}  and \ref{lem-5} as well as \eqref{ff} in Theorem \ref{th3} can be taken uniformly.
	
	In particular, $\int_\Omega|\nabla\varphi^m|^2dx$ and $\max_\Omega |\nabla\varphi^m|$ are uniformly bounded.
	
	For any given $\delta>0$, using \eqref{4.8} we obtain for sufficiently large $m$,
	$$
	\left|I\left(\varphi^m, q^m_\infty\right)-I\left(\varphi_m, \bar{q}_\infty\right)\right|<\delta,
	$$
	$$
	\left|I\left(\bar{\varphi}, q^m_\infty\right)-I\left(\bar{\varphi}, \bar{q}_\infty\right)\right|<\delta,
	$$
	combined with the minimality of $\phi^m$ for $I\left(\phi, q^m_\infty\right)$, we have
	\begin{equation*}
			I\left(\varphi^m, \bar{q}_\infty\right)\leq I\left(\varphi^m, q^m_\infty\right)+\delta
			\leq I\left(\bar{\varphi}, q^m_\infty\right)+\delta
			\leq I\left(\bar{\varphi}, \bar{q}_\infty\right)+2\delta.
	\end{equation*}
	Therefore, $\varphi^m$ is a minimizing sequence for $I\left(\varphi, \bar{q}_\infty\right)$. By the proof of Theorem \ref{th3}, $\varphi_m\rightarrow \bar{\varphi}$ in $\mathcal{V}$.	The uniform convergence of $\nabla\phi^m$ to $\nabla\bar{\phi}$ follows from those uniform estimates in Lemma \ref{lem-3} , \ref{lem-4} and \ref{lem-5} by employing Arzela--Ascoli  theorem and a contradiction argument.
		
	Then, we conclude that $\max_{x\in\Omega}|\nabla \phi^m|\rightarrow \max_{x\in\Omega} |\nabla \bar{\phi}|$, hence $\max_{x\in\Omega} |\nabla\phi|$ is a continuous function of
	$q_\infty$.
\end{proof}	
	
\textbf{Proof of Theorem \ref{mainthm}}:	

First, we will prove the part (1): the subsonic case.

Up to now, we have shown for fixed cut off parameter $\theta$, there exists an unique solution of \textbf{Problem 3 ($q_\infty$)}, which is denoted as $\phi(x; q_\infty, \theta)$. For remove the cut off, which is introduce in Section 3, we define the quantity:
$$
\mathcal{M}(q_\infty, \theta)=\max_{x\in\Omega}\left(\frac{|\nabla\phi(x; q_\infty, \theta)|}{q_{cr}(\psi)(x)}\right)
$$
which is equivalence of maximum Mach number of the field. It is noticeable that for certain $\theta$, if $\mathcal{M}(q_\infty, \theta)<1-2\theta$, $\phi(x; q_\infty, \theta)$ is the unique solution of \textbf{Problem 2 ($q_\infty$)}. By the similar argument in Theorem \ref{6.1}, one can show that $\mathcal{M}(q_\infty, \theta)$ also depends on $q_\infty$ continuously.

Let $\{\theta_i\}_{n=1}^{\infty}$ be a strictly decreasing sequence of positive numbers, such that $\delta_i\rightarrow0$ as $i\rightarrow\infty$. For fixed $i$, there exists a maximum interval $[0, q_\infty^i)$ such that, for $q_\infty\in[0, q_\infty^i)$,
$$
\mathcal{M}(q_\infty, \theta)<1-2\theta_i.
$$
Then, for $q_\infty\in[0, q_\infty^i)$, $\nabla \phi(x; q_\infty, \theta_i)$ is the solution of \textbf{Problem 1 ($q_\infty$)}. From the uniqueness of \textbf{Problem 3 ($q_\infty$)}, we can see $q_\infty^i\le q_\infty^j$ for $i<j$. So, $\{q_\infty^i\}_{i=1}^{\infty}$ is an increasing sequence with the upper bounded $(q_\infty)_{max}:=c(1)$, which implies the convergence of the sequence. As a consequence, we can have
$$\hat{q}:=\lim_{i\rightarrow\infty}q_\infty^i.$$
If $q_\infty^i<\hat{q}$ for any $i$, then for any $q_\infty\in [0, \hat{q})$, there exist an index $i$ such that $q_\infty\le q_\infty^i$. The solution of \textbf{Problem 1 ($q_\infty$)} is $\nabla\phi(x; q_\infty, \theta_i)$, which could be written as $\nabla\phi(x; q_\infty)$.

Then, we have $M(\nabla\phi(x; q_\infty); \psi)\rightarrow 1$ as $q_\infty\rightarrow \hat{q}$. It means the subsonic flows will become subsonic-sonic flows.

The uniqueness of \textbf{Problem 1 ($q_\infty$)} is already contained in Theorem \ref{6.1}.

Next, we will prove the part (2): the subsonic-sonic case.

The strong solutions $u^\varepsilon$ satisfy \eqref{p1equation}, and the Bernoulli's law \eqref{Brelation} and are uniform subsonic solutions of \textbf{Problem 1 ($q_\infty^\varepsilon$)}. Hence,
Theorem 2.2 in \cite{Chen-Huang-Wang} immediately implies the strong convergence of $u^\varepsilon$  in
$\Omega$. As a consequence, the density function $\rho^\varepsilon(x)$, which defined by \eqref{equ-rho}, is convergence to $\bar{\rho}(x)$. The boundary conditions are satisfied for $\bar{\rho} \bar{u}$
in the sense of Chen-Frid \cite{Chen7}. On the other hand, Since
$(\ref{OrE})_2$ holds for the sequence of subsonic solutions
$\rho^{\varepsilon}(x)$ and $u^{\varepsilon}(x)$, it is straightforward to see that $\bar{\rho}$ and $\bar{u}$ also
satisfies $(\ref{OrE})_2$ in the sense of
distributions. This completes the proof of Theorem \ref{mainthm}.

 {\bf Acknowledgements:}
The research of Tian-Yi Wang was supported in part
by NSF of China under Grant 11371064.


\begin{thebibliography}{99}
	
	
		\bibitem{Bers1}
		\newblock L. Bers,
		\newblock \emph{An existence theorem in two-dimensional gas dynamics,}
		\newblock Proc. Symposia Appl. Math., \textbf{1} (1949) 41--46.
		
		\bibitem{Bers2}
		\newblock L. Bers,
		\newblock \emph{Boundary value problems for minimal surfaces with singularities at infinity,}
		\newblock Trans. Amer. Math. Soc., \textbf{70} (1951) 465--491.
		
		\bibitem{Bers3}
		\newblock L. Bers,
		\newblock \emph{ Existence and uniqueness of a subsonic flow past a given profile.}
		\newblock Comm. Pure Appl. Math., \textbf{7} (1954) 441--504.
		
		\bibitem{Chen-Du-Xie-Xin}
		\newblock C. Chen, L. Du, C. Xie, and Z.P. Xin,
		\newblock \emph{Two Dimensional Subsonic Euler Flow Past a Wall or a Symmetric Body, }
		\newblock preprint, 2014, arXiv: 1410.1991.
		
		\bibitem{Chen6}
		\newblock G.-Q. Chen, C. M. Dafermos, M. Slemrod, and D.-H. Wang,
		\newblock \emph{On two-dimensional sonic-subsonic flow,}
		\newblock  Commun. Math. Phys., \textbf{271} (2007), 635--647.
		
		\bibitem{Chen7}
		\newblock G.-Q. Chen,  H. Frid,
		\newblock \emph{Divergence-measure fields and hyperbolic conservation laws,}
		\newblock  Arch. Rational. Mech. Anal., \textbf{147} (1999), 89--118.
		
		\bibitem{Chen-Huang-Wang}
		\newblock G.-Q. Chen, F.-M. Huang, and T.-Y. Wang,
		\newblock \emph{Subsonic-sonic limit of
		approximate solutions to multidimensional steady Euler equations,}
		\newblock Arch, Rational Mech. Anal., 2015 (to appear); arXiv:1311.3985.
	
		\bibitem{Courant-Friedrichs}
		\newblock R. Courant and K.O. Friedrichs,
		\newblock {\it Supersonic Flow and Shock Waves},
		\newblock Interscience Publishers Inc.: New York, 1948.
		
		\bibitem{Dong1}
		\newblock G.-C. Dong,
		\newblock \emph{Nonlinear partial differential equations of second
			order.}
		\newblock  American Mathematical Society, Providence, RI, 1991.
		
		
		\bibitem{Dong2}
		\newblock G.-C. Dong, and B. Ou,
		\newblock \emph{Subsonic flows around a body in space,}
		\newblock Comm. Partial Differential Equations, \textbf{18} (1993) 355--379.
		
		\bibitem{Du-Yan-Xin}
		\newblock L.Du, Z.P. Xin and W. Yan,
		\newblock \emph{Subsonic Flows in a Multi-Dimensional Nozzle,}
		\newblock Arch, Rational Mech. Anal., \textbf{201} (2011), 965--1012.
		
		
		\bibitem{Finn1}
		\newblock R. Finn, and D. Gilbarg,
		\newblock \emph{Asymptotic behavior and uniqueness of plane subsonic flows,}
		\newblock Comm. Pure Appl. Math., \textbf{10} (1957), 23--63.
		
		\bibitem{Gilbarg1}
		\newblock R. Finn, and D. Gilbarg,
		\newblock \emph{Three-dimensional subsonic flows and asymptotic estimates for elliptic partial differential equations,}
		\newblock Acta Math., \textbf{98} (1957) 265--296.
		
		\bibitem{Gilbarg-Trudinger}
		\newblock D. Gilbarg, and N. Trudinger
		\newblock \emph{Elliptic partial differential equations of second order}
		\newblock Springer-Verlag, New York,1983, second edition.

		
		\bibitem{Huang-Wang-Wang} F.-M. Huang, T.-Y. Wang, and Y. Wang,
		\newblock \emph{On multidimensional sonic-subsonic flow,}
		\newblock Acta Math. Sci. Ser. B, \textbf{31} (2011), 2131--2140.
		
		\bibitem{Liu-Yian}
		\newblock L. Liu and H. Yuan,
		\newblock \emph{Steady subsonic potential flows through infinite multi-dimensional largely-open nozzles,}
		\newblock  Calc. Var., \textbf{49} (2014) 1--36.
		
		\bibitem{LU_68}
		\newblock O.A. Ladyzhenskaya, and N.N. Ural'tseva,
		\newblock \emph{Linear and quasilinear elliptic equations,}
		\newblock  Academic Press, New York, 1968.

		
		\bibitem{ou2}
		\newblock G. Lu, and B. Ou,
		\newblock \emph{A Poincar\'{e} Inequality on $R^n$ and Its Application to Potential Fluid Flows in Space,}
		\newblock Comm. Appl. Nonlinear Anal., \textbf{12(1)} (2005) 1--24.
		
		\bibitem{Frankl}
		\newblock F. Frankl, and M.  Keldysh,
		\newblock \emph{Die ussere neumannshe aufgabe fr nichtlineare elliptische differentialgleichungen mit anwendung auf die theorie der flugel im kompressiblen gas.}
		\newblock Bull. Acad. Sci., \textbf{12} (1934) 561--687.
		
		\bibitem{ou1}
		\newblock B. Ou,
		\newblock \emph{An irrotational and incompressible flow around a body in space,}
		\newblock J. of PDEs, \textbf{7(2)} (1994) 160--170.
		
		\bibitem{Payne}
		\newblock L. E. Payne, and H. F. Weinberger,
		\newblock \emph{Note on a lemma of Finn and Gilbarg,}
		\newblock Acta Math., \textbf{98} (1957) 297--299.
		
		
		\bibitem{Shiffman1}
		\newblock M. Shiffman,
		\newblock \emph{On the existence of subsonic flows of a compressible fluid,}
		\newblock Proc. Nat. Acad. Sci. U.S.A., \textbf{38}	(1952) 434--438.
		
		\bibitem{Shiffman2}
		\newblock M. Shiffman,
		\newblock \emph{On the existence of subsonic flows of a compressible fluid,}
		\newblock   J. Rational Mech. Anal., \textbf{1} (1952) 605--652.
		
		\bibitem{Xin1}
		\newblock C. Xie and Z. Xin,
		\newblock \emph{Global subsonic and subsonic-sonic flows through infinitely long nozzles,}
		\newblock Indiana Univ. Math. J., \textbf{56} (2007), 2991--3023.
		
		\bibitem{Xin2}
		\newblock C. Xie and Z. Xin,
		\newblock \emph{Global subsonic and subsonic-sonic flows through infinitely long axially symmetric nozzles,}
		\newblock J. Diff. Eqs., {\bf 248} (2010), 2657--2683.
		
\end{thebibliography}
\end{document}